\documentclass[12pt]{amsart}
\usepackage{amssymb,amsthm,amsmath,amstext}
\usepackage{mathrsfs}  % allows nice script font for sheaves
\usepackage{bm}        % allows bold italic letters in math mode
\usepackage{mathtools} % allows more extendible arrows

\usepackage[left=1.28in,top=.8in,right=1.28in,bottom=.8in]{geometry}

\usepackage{graphicx}
\usepackage[all]{xy}

\theoremstyle{plain}
\newtheorem{theorem}{Theorem}[section]
\newtheorem{prop}[theorem]{Proposition}
\newtheorem{lemma}[theorem]{Lemma}
\newtheorem{cor}[theorem]{Corollary}
\newcommand{\Brp}{{}_p\mathrm{Br}}
\newcommand{\Brt}{{}_2\mathrm{Br}}
\newtheorem{theoremintro}{Theorem}

\newtheorem*{conj}{Conjecture} 

\theoremstyle{definition}

\newtheorem{example}[theorem]{Example}

\theoremstyle{remark}

\newcommand{\sheaf}[1]{\mathscr{#1}}

\newcommand{\OO}{\sheaf{O}}

\renewcommand{\AA}{\sheaf{A}}

\newcommand{\PP}{\sheaf{P}}
\newcommand{\XX}{\sheaf{X}}
\newcommand{\DD}{\sheaf{D}}
\newcommand{\QQ}{\sheaf{Q}}

\DeclareMathOperator{\br}{\mathrm{br}}

\usepackage{hyperref}
 
\hypersetup{colorlinks=true,linkcolor=blue,anchorcolor=blue,citecolor=blue,letterpaper=true}

\begin{document}

\title[Period-index and $u$-invariant questions  for function fields of curves]
{Period-index and $u$-invariant questions  for function fields   over complete discretely 
valued fields}

\author[Parimala]{R.\ Parimala }
\address{Department of Mathematics \& Computer Science \\ %
Emory University \\ %
400 Dowman Drive~NE \\ %
Atlanta, GA 30322, USA}
\email{ parimala@mathcs.emory.edu, suresh@mathcs.emory.edu}
 
\author[Suresh]{V.\ Suresh}

\date{}

\begin{abstract} Let $K$ be a complete discretely  valued field with residue field
$\kappa$ and $F$ the function field of a curve over $K$.  Let $p$ be the 
characteristic of $\kappa$ and $\ell$  a prime not equal to  $p$. 
If the Brauer $\ell$-dimensions of all finite extensions of $\kappa$ are bounded 
by $d$  and the Brauer $\ell$-dimensions of all extensions of $\kappa$ of
transcendence degree at most 1 are bounded  by $d+1$, then 
it is known that  the Brauer $\ell$-dimension 
of $F$ is at most $d + 2$ ([S1], [HHK1]). In this paper we give a bound for the Brauer 
$p$-dimension  of $F$ in terms of the $p$-rank of $\kappa$. As an application, 
we show that if $\kappa$ is a perfect field of
characteristic 2, then any quadratic form over $F$ in at least 9 variables is isotropic. 
If $\kappa$ is a finite field, this is a result of Heath-Brown/Leep ([HB], [Le]). 
\end{abstract} 
 
\maketitle

\def\ZZ{${\mathbf Z}$}
\def\ih{${\mathbf H}$}
\def\RR{${\mathbf R}$}
\def\IF{${\mathbf F}$}
\def\QQ{${\mathbf Q}$}
\def\IP{${\mathbf P}$}

Let $K$ be a field. For a central simple algebra $A$ over $K$, the
{\it period} of $A$ is the order of its class in the Brauer group  of
$K$ and the {\it index} of $A$ is the degree of the division algebra
 Brauer equivalent to $A$.  The index of $A$ is denoted by ind$(A)$
and period of $A$ by per$(A)$. Let $K$ be a $p$-adic field and $F $ 
the function field of a curve over $K$. The question whether the index of a 
central simple algebra over $F$ divides the square of its period 
has remained open for a while. For indices which are coprime to $p$,
an affirmative answer to this question is a theorem of Saltman ([S1]).  
To complete the answer, one needs to understand the relationship between 
the period and the index for algebras of period $p$ over $F$. One of the main results
proved in this paper is the following

\begin{theoremintro}
  Let $K$ be a $p$-adic  field
 and  $F$ a function field of a curve over $K$.   Then the 
 index of any central simple algebra over $F$
 divides the square of its period.
\end{theoremintro}

 Let $K$ be any field.  For a prime
$p$, we define the {\it Brauer $p$-dimension} of $K$, denoted by
Br$_p$dim$(K)$, to be the smallest integer $d \geq 0$ such that for
every finite extension $L$ of $K$ and for 
every central simple algebra $A$ over $L$ of period a power of $p$,
ind$(A)$ divides per$(A)^d$. The {\it Brauer dimension} of $K$,
denoted by Brdim$(K)$, is defined as the maximum of the Brauer 
$p$-dimension of $K$ as $p$ varies over all primes.  
Suppose  the characteristic of $K$ is $p > 0$.  If $[K : K^p ] = p^n$, then
$n$ is called the {\it $p$-rank} of $K$.  A field  of
characteristic $p > 0$ is perfect if and
only if its $p$-rank is 0. A theorem of Albert asserts that the  Brauer $p$-dimension
of a field $K$ of characteristic $p > 0$ is at most the $p$-rank of $K$ (cf. (1.2)). 

In this paper,  we   discuss more  generally the Brauer $p$-dimension
of   function fields  of  curves over a complete discretely valued field
of characteristic   0 with residue field of 
characteristic $p > 0$. 

We begin by bounding the Brauer dimension of complete discretely valued 
fields. Let  $K$ be a complete discretely valued field and  $\kappa$ its
residue field. Suppose that char$(K) = 0$ and char$(\kappa) = p > 0$. 
Let $\ell$ be a prime. Suppose that Br$_\ell$dim$(\kappa) \leq d$. If
$\ell \neq p$, then it  is well known that   
Br$_\ell$dim$(K) \leq d+1$ (cf. [GS], Corollary 7.1.10).  There  seems to be no good 
connections  between the Brauer $p$-dimension of $K$ and the
Brauer $p$-dimension of $\kappa$.  For any $n \geq 0$, we give an 
example of a complete discretely  valued field $K$ with Br$_p$dim$(K) \geq n$ 
and Br$_p$dim$(\kappa) = 0$.  However  there are bounds for  the
Brauer $p$-dimension of $K$ in terms of the $p$-rank of $\kappa$ and
we prove the following 

\begin{theoremintro}
  Let $K$ be a complete discretely valued field
with residue field $\kappa$. Suppose that char$(\kappa) = p > 0$
and the $p$-rank of $\kappa$ is $n$.  Then  Br$_p$dim$(K) \leq 1$ if $ n = 0$ 
and $\frac{n}{2} \leq $ Br$_p$dim$(K) \leq 2n$ if $n \geq 1$.
\end{theoremintro}

Let  $F$  be the function field of a curve over $K$.  Let  $\ell$ be a
prime. Suppose that
there exists $d $  such that Br$_\ell$dim$(\kappa) \leq d$
and Br$_\ell$dim$(\kappa(C) ) \leq d+1$ for every curve $C$ over $\kappa$.  
It was proved in [HHK1] that  if char$(\kappa) \neq \ell$, then Br$_\ell$dim$(F) \leq d + 2$.  
  For $\ell  = $char$(\kappa)$,  
we prove the following

\begin{theoremintro}
\label{thm:main1} Let $K$ be a complete discretely valued field
with residue field $\kappa$. Suppose that char$(K) = 0$ and
char$(\kappa) = p > 0$.  Let $F$ be the function field of a curve 
over $K$. If  the $p$-rank of $\kappa$ is $n$,
then Br$_p$dim$(F) \leq 2n+2$. 
\end{theoremintro}

We use the description of the Brauer group of a complete
discretely valued field in the mixed characteristic case due to Kato
([K], [CT]) and the patching techniques of Harbater-Hartman-Krashen
([HHK1]) to prove our main results.

In the last section, we derive some consequences for the $u$-invariant
of fields. The $u$-invariant of a field $L$ is the maximum dimension
of anisotropic quadratic forms over $L$. Let $K$ be a complete
discretely valued field with residue field $\kappa$. It is a theorem of
Springer that $u(K) = 2u(\kappa)$. Let $F$ be a function field of a
curve over $K$. Suppose that char$(\kappa) \neq 2$. If  $u(L) \leq d$ for every
finite extension $L$ of $\kappa$ and $u(\kappa(C)) \leq 2d$ for every
function field $\kappa(C)$ of a curve $C$ over $\kappa$, then in
([HHK1]), it is proved that $u(F) \leq 4d$. For a $p$-adic field $K$,
this was proved in ([PS2]).     
Suppose $\kappa$ is a field of characteristic 2
with $[\kappa : \kappa^2] = n$. Then
 $u(\kappa) \leq 2n$ ([MMW], Corollary 1).  Let $L$ be a finite extension
of $\kappa$. Since $[L : L^2] = n$, we have $u(L)
\leq 2n$.  If  $C$ is a curve over $\kappa$, then
$[\kappa(C) : \kappa(C)^2]  = 2n $ and hence $u(\kappa(C))
\leq 4n$.   If  char$(K) = 2$,  $[F : F^{*2}]  = 4n$ and hence $u(F) \leq 8n$
([MMW]. Corollary 1). 
Suppose that char$(K) =0$.
If $\kappa$ is a finite field, then results of Heath-Brown ([HB) and Leep ([Le])
lead to $u(F) = 8$. 
However very little is known about $u(F)$ for general $\kappa$.  We prove  
the following

\begin{theoremintro}
Let $K$ be a complete discretely valued field with residue field $\kappa$.
Suppose that char$(K) = 0$ and char$(\kappa) = 2$. Let $F$ be a function field
of a curve over $K$. If $\kappa$ is a perfect field, then $u(F) \leq 8$.
\end{theoremintro}

This leads us to  the following
 
\begin{conj}  Suppose $K$ is a complete discretely valued field of
characteristic 0 with residue field $\kappa$ of characteristic 2. 
If $F$ is a function field of a curve over $K$, then $u(F) \leq
8[\kappa: \kappa^2]$. 
\end{conj}
 
\bigskip

{\small\noindent{\bf Acknowledgements.}  %
  We thank Asher Auel for his several critical comments on the text.}

\section{ Module of differentials and Milnor $k$-groups} 

We begin by recalling two well-known results (1.1, 1.2) concerning
the Brauer $\ell$-dimension of a field.   Lemma 1.1 reduces the
computation of the Brauer $\ell$-dimension to bounding indices of prime
exponent algebras. Corollary 1.2 computes the Brauer $p$-dimension for
fields of characteristic $p>0$.  

\begin{lemma} Let $k$ be any  field and $\ell$ a prime. If for
every central simple algebra $A$ of period $\ell$ over a finite extension
$K$ of $k$, ind$(A)$ divides $\ell^d$, then Br$_\ell$dim$(k) \leq  d$. 
\end{lemma}

\begin{proof} Let $k'$ be a finite extension of $k$ and  $A$ 
a central simple algebra over $k'$ of
period $\ell^n$ for some $n$. We prove by induction on $n$ that ind$(A)$
divides $(\ell^n)^d$. The case $n= 1$ is the given hypothesis. Suppose
that  the lemma holds for   $n-1$. Let $A' = A^{\otimes \ell}$. Then per$(A') = \ell^{n-1}$.
By the induction hypothesis  ind$(A')$ divides $(\ell^{n-1})^d$. 
Thus there exists a finite extension $K$ of $k'$ of degree
dividing $(\ell^{n-1})^d$ such that  $A' \otimes_{k'} K$ is a matrix
algebra.   In particular per$(A \otimes_{k'} K) = \ell$ and by the
hypothesis ind$(A \otimes _{k'} K)$ divides $\ell^d$. 
 Thus  there exists a  finite extension  $L$ of $K$  of degree dividing $\ell^d$  such
that $A \otimes_{k'} L$ is a matrix algebra.  Since  $[L : k'] = [L : K][K : k']$
divides $\ell^d(\ell^{n-1})^d = (\ell^n)^d$,  ind$(A)$ divides $(\ell^n)^d$. 
\end{proof}

\begin{cor}(Albert) Let $\kappa$ be a field of characteristic $p > 0$. 
Suppose that the $p$-rank of $\kappa$ is $n$.  Then Br$_p$dim$(\kappa) \leq n$. 
\end{cor}

\begin{proof}   Let $k'$ be a finite extension of $k$ and $A$ be    
a central simple algebra over $k'$ of period $p$. 
By (1.1),  it is enough to show that  ind$(A)$ divides $p^n$. 
By ([A], Chap. VII. Theorem 32),  $A
\otimes_{k'} k'^{1/p}$ is  a matrix algebra and hence ind$(A)$ divides $[k'^{1/p} : k']$. 
Since  $[k'^{1/p} : k'] = [k' : k'^p] = [k  : k^p] =  p^n$  ([B], A.V.135, Corollary 3), 
  ind$(A)$ divides $p^n$. 
\end{proof}

\vskip 3mm

Let $\kappa$ be a field of characteristic $p > 0$. 
Let $\Omega^1_\kappa$ be the module of differentials on $\kappa$. Then
the dimension of $ \Omega_\kappa^1$ as a $\kappa$-vector space is equal to
the  $p$-rank of $\kappa$.   
Let   $\Omega^2_\kappa$ be the second exterior power of
$\Omega^1_\kappa$.  
Let $K_2(\kappa)$ be the Milnor  $K$-group and $k_2(\kappa) =
K_2(\kappa)/pK_2(\kappa)$.  Then there is an injective homomorphism  (cf.,
[CT], 3.0) 
$$h^2_p : k_2(\kappa) \to
\Omega^2_\kappa
$$ 
given by 
$$(a, b) \mapsto \frac{da}{a}\wedge
\frac{db}{b}.$$

Suppose $\kappa = \kappa^p(a_1, \cdots , a_n)$.  Then every element in
$\Omega^2_\kappa$ is a linear combination of elements $da_i \wedge da_j$.
In fact if $a_1, \cdots , a_n$ is a $p$-basis of $\kappa$, then $da_i
\wedge da_j$, $1 \leq i < j \leq n$ is a basis of $\Omega^2_\kappa$ over
$\kappa$. 

\vskip 3mm

We now record a few facts about  $\Omega_\kappa^2$
and $k_2(\kappa)$.

\begin{lemma} 
Let $a, b \in \kappa^*$ be $p$-dependent. Then $(a, b) = 0  \in  k_2(\kappa)$. 
\end{lemma}

\begin{proof}  If $a$ is a $p^{\rm th}$ power in $\kappa$, then $da = 0$. 
Suppose $a  = \sum \lambda_i^p b^i$  
for some $\lambda_i \in \kappa$.  Then   $da =
(\sum \lambda_i^pib^{i-1})db$  
and $da\wedge db = 0$.  In particular  $\frac{da}{a}  \wedge
\frac{db}{b}  = 0   \in \Omega^2_\kappa$. 
Since $h^2_p((a,b)) =  \frac{da}{a} \wedge \frac{db}{b}$ and $h^2_p$ is injective, 
we have $(a, b) = 0  \in  k_2 (\kappa)$.  
\end{proof}

\begin{lemma} Suppose that $\kappa = \kappa^p(a_1, \cdots
a_n)$. Then the natural homomorphism $k_2(\kappa) \to
k_2(\kappa(\sqrt[p]{a_1}, \cdots , \sqrt[p]{a_{n-1}}))$ is zero. 
\end{lemma}

\begin{proof}  Let $(a, b ) \in k_2(\kappa)$.  Let $\kappa' =
\kappa(\sqrt[p]{a_1}, \cdots , \sqrt[p]{a_{n-1}})$. 
If $a$  is a $p^{\rm th}$ power in $\kappa'$, 
then the image of $(a, b ) \in
k_2(\kappa' )$ is zero.  Suppose
that  $a$ is not a $p^{\rm th}$ power  in
$\kappa'$. Since $\kappa'^p = \kappa^p(a_1, \cdots , a_{n-1})$,  $\kappa =
\kappa'^p(a_n)$ and hence $[\kappa : \kappa'^p] \leq p$. Since $a \not\in
\kappa'^p$, we have $\kappa = \kappa'^p(a)  = \kappa^p(a_1, \cdots , a_{n-1}, a)$.
In particular $a$ and $b$ are $p$-dependent over $\kappa'$ and hence, 
by (1.3),   $(a, b) = 0 \in k_2(\kappa')$.   Since every element
in $k_2(\kappa)$ is  a sum of elements of the form $(a, b)$, the image
of $k_2(\kappa)$ in $k_2(\kappa')$ is zero.
\end{proof}

\begin{lemma} Let $a_1, \cdots , a_n \in \kappa^*$ be  
  $p$-independent over  $\kappa$ and $\kappa' =  \kappa(\sqrt[p^{r_1}]{a_1},
\cdots , \sqrt[p^{r_n}]{a_n})$. Then every element in the  kernel of
the natural homomorphism $\Omega_\kappa^2 \to \Omega_{\kappa'}^2$ is
of the form $da_1 \wedge f_1 + \cdots + da_n \wedge f_n$ for some
$f_i \in \Omega_\kappa^1$. 
\end{lemma}

\begin{proof} Let $B \subset  \kappa^*$ be such
that $B \cap \{ a_1, \cdots a_n \} = \emptyset$ and $B \cup \{ a_1,
\cdots , a_n \}$ is a $p$-basis of $\kappa$.  Then $B \cup \{
\sqrt[p^{r_1}]{a_1}, \cdots , \sqrt[p^{r_n}]{a_n} \}$ is  a $p$-basis
of $\kappa'$.  Let $\alpha \in \Omega^2_\kappa$ be in the kernel of 
$\Omega^2_\kappa \to \Omega^2_{\kappa'}$.  Then $\alpha = \sum
\lambda_{ij} da_i \wedge da_j$ with $1 \leq i < j\leq m$ and 
$a_{n+1}, \cdots , a_m \in B$, $\lambda_{i} \in \kappa$. 
Since the image of $da_i \wedge da_j$ in $\Omega^2_{\kappa'}$ is
zero for $1 \leq i \leq n$ and  the image
of $\alpha$ in $\Omega_{\kappa'}^2$ is zero, the image of 
 $\sum \lambda_{ij} da_i \wedge da_j$,   $n+1 \leq i < j \leq m$, in $\Omega^2_{\kappa'}$
 is zero.  Since $B$ is $p$-independent over
$\kappa'$ and $a_{n+1},  \cdots , a_m \in B$,  $da_{i} \wedge da_j$, $
n+1 \leq i < j \leq m$  are linearly independent over $\kappa'$ and hence
$\lambda_{ij}  = 0$  for $n+1 \leq i < j
  \leq m$. Thus   $\alpha = \sum \lambda_{ij} da_i \wedge 
da_j$ with $1 \leq i < j \leq n$.  
\end{proof}

\begin{lemma} Let  $a_1, \cdots ,a_{2n} \in \kappa^*$   be  $p$-independent in
$\kappa$.  Let $\kappa'$ be an extension of $\kappa$ of degree $d$ and
$\lambda_1, \cdots , \lambda_n \in \kappa^*$. If
$d < p^{n}$, then the image of $\lambda_1da_1 \wedge
 da_2  +  \cdots + 
\lambda_n da_{2n-1}  \wedge  da_{2n} $  in
$\Omega_{\kappa'}^2$ is non-zero.  
\end{lemma}

\begin{proof}  Since $\Omega_\kappa^2 \to \Omega_{\kappa_1}^2 $
is injective for any separable extension $\kappa_1$ of $\kappa$, by
replacing $\kappa$ by the separable closure of $\kappa$ in $\kappa'$,
we assume that $\kappa'$ is purely inseparable over $\kappa$. Then
$\kappa' = \kappa(\sqrt[p^{r_1}]{b_1},
\cdots , \sqrt[p^{r_n}]{b_m})$ for some $b_i \in \kappa^*$ with $\{
b_1, \cdots , b_m \}$ $p$-independent over $\kappa$.  Since the kernels
of the homomorphisms $\Omega_\kappa^2 \to \Omega^2_{\kappa'}$ and
$\Omega^2_\kappa \to \Omega^2_{\kappa(\sqrt[p]{b_1}, \cdots
  ,\sqrt[p]{b_m} )}$ are equal by (1.5), we assume that $\kappa' =
  \kappa(\sqrt[p]{b_1}, \cdots ,\sqrt[p]{b_m})$. Since $[\kappa' :
  \kappa] = p^m < p^n$, we have $m \leq n-1$. Without loss of
  generality we assume that $m = n-1$. 

Suppose that $\{ a_1, \cdots , a_{2n} \}$ is a $p$-basis of
$\kappa$.  Let $r$ be the maximum such that \\ $\{b_1, \cdots , b_{n-1},
a_{i_1},  a_{i_2} \cdots , a_{i_r} \}$ is $p$-independent with no two
$a_{i_s}$ in $\{ a_{2j-1}, a_{2j} \}$. By
reindexing, we assume that $\{b_1, \cdots , b_{n-1},   a_1, a_3, \cdots ,
a_{2r+1}  \} $ is $p$-independent. Then, for each $i \geq 2r+3$,  
$\{b_1, \cdots , b_{n-1},   a_1, a_3, \cdots ,
a_{2r+1} , a_i \}$ is $p$-dependent.   
Since $\{a_1, \cdots ,  \\ a_{2n} \}$ is $p$-basis of $\kappa$, there
exists $1 \leq t_1 <  t_2 < \cdots  <t_q \leq {r+1}  $ such that 
$\{b_1, \cdots , b_{n-1}, a_1, a_3, \cdots , a_{2r + 1}, a_{2t_1}, 
\cdots , a_{2t_q} \}$ is a $p$-basis of $\kappa$.  After reshuffling
the indices, we assume that $t_1 = 1, \cdots, t_q = q$ and $B = \{b_1,
\cdots , b_{n-1},  a_1,   a_3,  \cdots,  a_{2r+1}, a_2, \\ a_4,  \cdots  a_{2q},
  \}$ is a  $p$-basis of $\kappa$ with $q \leq r+1$.  Then $\{ \sqrt[p]{b_1}, \cdots , 
  \sqrt[p]{b_{n-1}},  a_1, a_3, \cdots, a_{2r+1}, \\ a_2, a_4, \cdots , a_{2q} \}$ is a
  $p$-basis of $\kappa'$Ê

Since $B $ is a $p$-basis of $\kappa$,  every
element of $\Omega^2_{\kappa}$ can  be written as a linear combination
of $dx \wedge dy$,  $x,  y \in B$. We now  compute the coefficient of
$da_1 \wedge 
da_2$ in the expansion of $da_{2i+1} \wedge da_{2i+2}$ as a linear
combination of  $dx \wedge dy$,  $x, y \in B$.
Let $1 \leq i \leq r$.   Since $a_{2i+1}  \in B$,  
the coefficient of $da_1 \wedge da_2$ in
the expansion of $da_{2i+1} \wedge da_{2i+2}$ is zero.  
Let $i > r $.  Since $a_{2i+1}$ and $a_{2i+2}$ are  $p$-dependent
over $\{ b_1, \cdots , b_{n-1}, a_1, a_3, \cdots ,a_{2r+1} \}$, 
 in the expansion of $da_{2i+1}$ and $da_{2i+2}$ there is no
$da_2$ term. Hence, there is no $da_1\wedge da_2$ term in the
expansion of $da_{2i+1} \wedge da_{2i+2}$.  

Thus, the coefficient of $da_1 \wedge
da_2$ in the expansion of $ \alpha = \lambda_1  da_1\wedge
 {da_2}  +  \cdots + 
\lambda_n da_{2n-1}  \wedge  da_{2n} $ as a linear
combination of $dx \wedge dy$ with $x,  y \in B$ is  $\lambda_1 $.
Since $\{ \sqrt[p]{b_1}, \cdots , 
  \sqrt[p]{b_{n-1}},  a_1, a_3, \cdots, a_{2r+1}, a_2, a_4, \cdots , a_{2q} \}$ is a
  $p$-basis of $\kappa'$, the image of  $\alpha$ in $\Omega_{\kappa'}^2$ is
  non-zero. 

Let $\{ a_1, \cdots , a_{2n} \}$ be any $p$-independent subset of
$\kappa$. Let $B' \subset \kappa$ be such that $B' \cup \{ a_1, \cdots ,
a_{2n} \}$ is a $p$-basis of $\kappa$ and $B' \cap \{ a_1, \cdots ,
a_{2n} \} = \emptyset$. Let $\kappa_1$ be the extension of $\kappa$
obtained by adjoining $\sqrt[p^d]{b}$ for all $b \in B'$ and $d \geq
1$. Then $\{a_1, \cdots, a_{2n} \}$ is a $p$-basis of
$\kappa_1$. Then $\kappa_1\kappa'$ is an extension of $\kappa_1$ of
degree $< p^n$. Hence the image of $ \lambda_1dc_1 \wedge dc_2  + \cdots + 
 \lambda_n dc_{2n-1}  \wedge  dc_{2n}$ in
$\Omega_{\kappa'\kappa_1}^2$ is non-zero. In particular, the image of 
$ \lambda_1 dc_1 \wedge  dc_2  + \cdots + 
\lambda_n dc_{2n-1}  \wedge  dc_{2n} $ in
$\Omega_{\kappa'}^2$ is non-zero.
 \end{proof}

\section{Brauer $p$-dimension of a complete discretely valued field}

In the section we give a bound for    the Brauer $p$-dimension of a complete discrete
valued field of characteristic zero with residue of characteristic $p> 0$,  in terms of the
$p$-rank of the residue field.

\vskip 3mm

Let $R$ be a complete discrete valuation  ring with field of fractions $K$
and residue field $\kappa$. Let $\nu$ be the discrete valuation on $K$
given by $R$ and $\pi$ a parameter in $R$. Suppose that char$(K) = 0$ and
char$(\kappa) = p>0$ and that $K$ contains a primitive $p^{th}$ root
of unity.   Fixing a primitive $p^{\rm th}$ root of unity $\zeta \in
K^*$,  for $a, b \in K^*$,
let $(a, b) \in { \Brp}(K)$ be the class of the cyclic  
$K$-algebra  defined by $x^p = a,  y^p = b$ and $xy = \zeta yx$.
Let $N = \nu(p)p/(p-1)$.  Let $\br(K)_0  = {{\Brp}}(K)$. For $i \geq 1$, let $U_i = \{ u \in
R^* \mid u \equiv 1 ~{\rm mod}~ (\pi)^i \}$ and  $\br(K)_i$ be the subgroup
of $\Brp(K)$ generated by cyclic algebras $(u, a)$ with $u \in U_i$
and $a \in K^*$.  Since $R$ is complete,  for $n > N$,  every element in $U_n$ is
a $p^{\rm th}$ power and   $\br(K)_{n} = 0$.   

Let $\Omega^1_\kappa$ be the module of differentiales on
$\kappa$. For any $c \in \kappa$, let  $\tilde{c} \in
R$ be a lift of $c$.  For $i \geq 1$,  the maps 
$$  \Omega_\kappa^1   \to  \br(K)_i/\br(K)_{i+1}$$
given by $x\frac{dy}{y} \mapsto (1 + \tilde{x}\pi^i, \tilde{y})$ and 
$$  \kappa \to  \br(K)_i/\br(K)_{i+1}$$
given by $z \mapsto (\pi, 1+ \tilde{z}\pi^i)$ yield a surjective homomorphism
$$ \rho_i  : \Omega_\kappa^1 \oplus \kappa \to  \br(K)_i/\br(K)_{i+1} $$
([K], Thm. 2, cf. [CT], 4.3.1).

Let $K_2(\kappa)$ be the Milnor  $K$-group and $k_2(\kappa) =
K_2(\kappa)/pK_2(\kappa)$.  
The maps 
$$ k_2(\kappa)    \to \br(K)_0/\br(K)_1$$
given by $ (x, y)  \mapsto (\tilde{x}, \tilde{y})$ and  
$$ \kappa^*/\kappa^{*p} \to \br(K)_0/\br(K)_1$$
given by $(z)  \mapsto (\pi, \tilde{z})$ yield  an  isomorphism 
$$\rho_0 : k_2(\kappa)  \oplus \kappa^*/\kappa^{*p} \to \br(K)_0/\br(K)_1$$
([K], Thm.2, cf. [CT], 4.3.1).

\begin{lemma}  Let $R$, $K$ and $\kappa$ be as above.
Suppose  that $\kappa = \kappa^p(a_1, \cdots ,a_n)$ for some $a_1,
\cdots , a_n  \in \kappa$.  Let $u_1, \cdots , u_n \in R$ be lifts of
$a_1, \cdots ,  a_n$.  Let $\alpha \in {\Brp}(K)$.  Then, there
exists   $u \in R^* $ such that  $(\alpha  - (\pi, u))  \otimes 
K(\sqrt[p]{u_1}, \cdots ,\sqrt[p]{u_{n-1}} ) \in  \br(K(\sqrt[p]{u_1},
\cdots ,\sqrt[p]{u_{n-1}}))_1$. 
\end{lemma}

\begin{proof}  Since  $\rho_0$ is surjective, there exists $x_i,
y_i, a \in \kappa^*$ such that $\rho_0( \sum_i (x_i, y_i) - (a) )
= \alpha  \in \br(K)_0/ \br(K)_1$.  In particular,  if $u$ is a lift of $a$ in $R$, 
$$ \rho_0 (  \sum_i(x_i,  y_i)) = \alpha - (  \pi, u) \in \br(K)_0/\br(K)_1.$$
Let $K' = K(\sqrt[p]{u_1}, \cdots , \sqrt[p]{u_n-1})$. Then $K'$ is also a complete discretely 
valued field with residue field $\kappa' = \kappa(\sqrt[p]{a_1}, \cdots ,
\sqrt[p]{a_{n-1}})$.  
By the functoriality of the map $\rho_0$,  
we have  
$\rho_0(\sum(x_i, y_i) )  =  \alpha - ( \pi, u)  \in  \br(K')_0/\br(K')_1$. 
By (1.4), the image of $\sum_i(x_i, y_i) $ is zero in $k_2(\kappa')$. Hence
$\alpha - (\pi, u)  = \rho_0(\sum(x_i, y_i) )  = 0 \in \br(K')_0/\br(K')_1.$ In particular, 
$\alpha - (\pi, u)  \in \br(K')_1$.  
\end{proof}

\begin{prop} Let $R$, $K$, $\kappa$ and $\pi$ be as above. 
Suppose  that $\kappa = \kappa^p(a_1, \cdots ,a_n)$ for some $a_1,
\cdots , a_n  \in \kappa$.  Let $u_1, \cdots , u_n \in R$ be lifts of
$a_1, \cdots ,  a_n$.  Let $\alpha \in \br(K)_1$.  
Then  there exist $\lambda,  \lambda_1, \cdots , \lambda_n \in R^*$ such that 
$$ \alpha = (\lambda_1, u_1) + \cdots  +  (\lambda_n, u_n)  +  (\pi, \lambda).$$ 
\end{prop}
 
\begin{proof}  Let $\alpha \in \br(K)_1$.
First we show, by induction on $i$,  that for each  $i \geq 0$,  
there exist $x_{i1},  \cdots  , x_{in}, x_i \in R^*$ 
such that  $\alpha - ( x_{i1}, u_1)  - \cdots (x_{i n}, u_n) - (\pi, x_i ) \in \br(K)_{i+1}$. 
 Since $\alpha \in \br(K)_1$, we take $x_{0j} = x_0  =1$, $1 \leq j
 \leq n$.  
Suppose that $i \geq 1$ and  there exist $x_{i1},  \cdots  , x_{in}, x_i \in R^*$ 
such that  $\alpha - ( x_{i1}, u_1)  - \cdots ( x_{in}, u_n) - ( \pi, x_i) \in \br(K)_{i+1}$.
Since the homomorphism  $\rho_{i+1}  :  \Omega^1_\kappa  \oplus \kappa
\to \br(K)_{i+1}/ \br(K)_{i+2}$
 is surjective, there exist $w \in \Omega^1_{\kappa}$ and $a \in \kappa$ such that 
 $$\rho_{i+1}(w ,  a) = \alpha - ( x_{i1}, u_1)  - \cdots (x_{in}, u_n)
 - (\pi, x_i ) \in  \br(K)_{i+1}/\br(K)_{i+2}.$$
 Since $\kappa = \kappa^p(a_1, \cdots ,a_n)$,
$\Omega_\kappa$ is generated by $\frac{da_i}{a_i}$, $1 \leq i \leq n$ and   hence
$w = \sum_i b_i \frac{da_i}{a_i}$ for some $b_i \in \kappa$.
Thus   
$$ \rho_{i+1}(w ,  a) = (1 + \tilde{b}_1\pi^{i+1}, u_1) + \cdots (1 +
\tilde{b}_n\pi^{i+1}, u_n) +  (\pi, 1 + \tilde{a}\pi^{i+1}).$$
In particular,  
$\alpha - (x_{i1}, u_1)  - \cdots -  (x_{in}, u_n) - (\pi, x_i )   - (1 + \tilde{b}_1\pi^{i+1}, u_1) -
 \cdots -(1 + \tilde{b}_n\pi^{i+1}, u_n) -  (\pi, 1 + \tilde{a}\pi^{i+1})  \in \br(K)_{i+2}.$
Let $x_{(i+1)j} =  x_{ij}(1 + \tilde{b}_j\pi^{i+1})$ for $ 1 \leq j
\leq n$ and $x_{i+1} =  x_i(1 +  \tilde{a}\pi^{i+1})$. 
Since $(x, yz) = (x, y)  + (x, z) \in \Brp(K)$, we have
$\alpha - (x_{(i+1)1}, u_1) -  \cdots -(x_{(i+1)n}, u_n) - (\pi, x_{i+1}) \in \br(K)_{i+2}$. 
 
Since $\br(K)_{i} = 0$ for $i > N$, we have $\alpha = (x_{(N+1)1},
u_1) +  \cdots +  (x_{(N+1)n}, u_n) + (\pi, x_{N+1})$.
 \end{proof}

\begin{cor}  Let $K$ and $\kappa$ be as above. 
Suppose that the $p$-rank of $\kappa$ is $n$. Let $D$ be a central simple algebra 
over $K$ of period $p$.  If $D$ represents an element in $\br(K)_1$, then ind$(D)$ divides $p^{n+1}$.
\end{cor}

\begin{proof} Let $\alpha \in \br(K)_1$ be the class of $D$.  By (2.2),
there exist $\lambda, \lambda_1,  \cdots, \lambda_n \in R^*$ such that 
$\alpha = (\lambda_1, u_1) + \cdots + (\lambda_n, u_n) + (\pi, \lambda)$. 
Hence $\alpha \otimes K(\sqrt[p]{u_1}, \cdots \sqrt[p]{u_n}, \sqrt[p]{\pi}) = 0$
and the index of $D$ divides $p^{n+1}$. 
\end{proof}

\begin{theorem}  Let  $K$ be a complete discretely valued field with residue field $\kappa$.
Let $R$ be the valuation ring of  $K$ and $\pi \in R$ be a parameter.  
Suppose that char$(K) = 0$, char$(\kappa) = p$  and   
the   $p$-rank of $\kappa$ is  $n$.  Let $a_1, \cdots , a_n \in \kappa$
be such that $\kappa = \kappa^p(a_1, \cdots , a_n)$ and $u_1, \cdots , u_n \in R$ be lifts of 
$a_1, \cdots , a_n$. Let $D$ be a central simple algebra  over $K$ of period $p$.
If $n = 0$, then  $D \otimes K(\sqrt[p]{\pi})$ is  a matrix algebra and if $n \geq 1$, then 
$D \otimes K(\sqrt[p^2]{u_1}, \cdots \sqrt[p^2]{u_{n-1}},  \sqrt[p]{u_n}, \sqrt[p]{\pi})$  is a matrix algebra.
\end{theorem}

\begin{proof}      Let $\zeta $ be a primitive $p^{\rm th}$ root of unity and $K' = K(\zeta)$.  Since 
$[K' : K]$ is coprime to $p$, ind$(D) = $ ind$(D \otimes K')$.  Since
$K'$ is finite extension of a 
complete discretely valued field $K$, $K'$ is also a complete discrete
valued field with 
residue field $\kappa'$ a 
finite extension of  $\kappa$.   In particular, $p$-rank$(\kappa') = p$-rank$(\kappa)$. 
Thus, by replacing, $K$ by $K'$, we assume that $K$ contains a primtiive $p^{\rm th}$ root of unity. 
Let $\alpha \in \Brp(K)$ be the class of $D$. 

Suppose $n = 0$.  Then $\kappa = \kappa^p$ and $k_2(\kappa) = 0$. 
Since   $\rho_0
: k_2(\kappa) \oplus \kappa^*/\kappa^{*p} \to \br(K)_0/\br(K)_1$ is an isomorphism, 
 $\Brp(K) = \br(K)_1$.    Thus, by (2.2),   $\alpha = ( \pi, u)$.
 In particular $D \otimes K(\sqrt[p]{\pi})$ is a matrix algebra.  

Suppose that  $n \geq 1$.  Since $p$-rank$(\kappa) = n$, there exist 
  $a_1, \cdots , a_n \in \kappa^*$  such that 
$\kappa = \kappa^p(a_1, \cdots , a_n)$. 
Let  $u_1, \cdots , u_n \in R$ be lifts of $a_1, \cdots , a_n$ and
$\pi \in R$ a parameter. 
Let $K_1 = K(\sqrt[p]{u_1}, \cdots ,\sqrt[p]{u_{n-1}})$.  
Then  $K_1$ is also a complete discrete
valued field   with  residue field $\kappa_1
= \kappa(\sqrt[p]{a_1}, \cdots , \sqrt[p]{a_{n-1}})$. 
Let $R_1$ be the valuation ring of  $K_1$. Then  
$\pi$ is a parameter in $R_1$.     By (2.1),  there exists $u \in R^*$ such that 
$(\alpha - ( \pi, u) ) \otimes K_1 \in
\br(K_1)_1$.   Since  $ \kappa_1^p =  \kappa^p(a_1, \cdots ,
a_{n-1})$, we have  $\kappa_1 = \kappa_1^p(\sqrt[p]{a_1}, \cdots, \sqrt[p]{a_{n-1}}, a_n)$.  
Since $\alpha - (\pi,  u) \in \br(K_1)_1$, 
by (2.2), there exist  $\lambda_1, \cdots , \lambda_n, \lambda \in R_1$ such that 
$$
\alpha - (\pi, u)    =  (\lambda_1, \sqrt[p ]{u_1}) + \cdots + (\lambda_{n-1}, \sqrt[p ]{u_{n-1}}) + 
(\lambda_n, u_n) + (\pi, \lambda). 
$$
Hence 
$$\alpha  =   (\lambda_1, \sqrt[p]{u_1}) + \cdots + (\lambda_{n-1}, \sqrt[p]{u_{n-1}}) + 
(\lambda_n,  u_n) + (\pi, u\lambda).
$$ 
In particular $D \otimes K(\sqrt[p^2]{u_1}, \cdots
\sqrt[p^2]{u_{n-1}},  \sqrt[p]{u_n}, \sqrt[p]{\pi})$   is a matrix algebra.
\end{proof} 
 
\begin{cor} Let $K$, $\kappa$  and $n$ be as in (2.4).  
Then  Br$_p$dim$(K)$ is 1 if $ n = 0$ and Br$_p$dim$(K) \leq  2n$ if $n \geq 1$.
\end{cor}

\begin{proof} Let $K'$ be a finite extension of $K$.  Let $D$ be a central simple algebra
over $K'$ of period $p$.   
Since $K'$ is also a complete discretely valued field with the $p$-rank of the residue field equal to $n$,
corollary follows by (2.4) and (1.6). 
\end{proof}

\begin{lemma} Let $K$ be a complete discretely valued field
with residue field $\kappa$. Suppose that char$(K) = 0$,
char$(\kappa) = p>0$ and $[\kappa : \kappa^p] \geq 2n$.  Then
Br$_p$dim$(K) \geq n$.
\end{lemma}

\begin{proof} Let $a_1, \cdots , a_{2n} \in \kappa^*$ be
$p$-independent. Let  $u_1, \cdots , u_{2n} \in K^*$ be the lifts of
$a_1, \cdots , a_{2n}$.  Let $D = (u_1,  u_2) + \cdots + 
(u_{2n-1}, u_{2n})$. We claim that ind$(D) = p^{n}$. This would show
that Br$_p$dim$(K) \geq n$. 
 
Let $K_1$ be an   extension   of $K$ of degree at most  $p^{n-1}$.
 Since $K$ is a complete discretely valued field, $K_1$ is
also a complete discretely valued field with  residue field $\kappa_1$
and   $[\kappa_1 : \kappa] \leq [L : K] \leq p^{n-1}$.  Then, by (
1.6), the image of  $da_1 \wedge  da_2 +  \cdots +
da_{2n-1}\wedge da_{2n} $ in $ \Omega^2_{\kappa_1}$ is non-zero.  Since
 $h^2_p((a_1, a_2) + \cdots + (a_{2n-1}, a_{2n})) = da_1 \wedge  da_2 +
\cdots +
da_{2n-1}\wedge da_{2n} $ is nonzero in   $ \Omega^2_{\kappa_1}$,
$(a_1, a_2) + \cdots + (a_{2n-1}, a_{2n})$ is non-zero in
$k_2(\kappa_1)$.   Since $\rho_0((a_1, a_2) + \cdots + (a_{2n-1},
a_{2n}))$ is the class of $D \otimes_K K_1$ in $\br(K_1)_0/\br(K)_1$
and $\rho_0$ is injective,   $D
\otimes_K K_1$ is non-trivial in $\Brp(K_1)$. Hence ind$(D) \geq p^n$. Since
$D$ is a product of $n$ cyclic algebras, ind$(D) = p^n$.  
\end{proof}

Combining (2.4) and (2.6), we have the following

\begin{theorem} Let $K$ be a complete discretely valued field with residue field $\kappa$.
Suppose that char$(K) = 0$, char$(\kappa) = p > 0$ and  the $p$-rank of $\kappa$ is $n$.
If $n = 0$, then $Br_pdim(K) \leq 1$ and if $n \geq 1$, then $\frac{n}{2} \leq Br_pdim(K) \leq 2n$. 
 \end{theorem}

\begin{example} Let $k$ be  a purely  transcendental extension of
the finite field ${\mathbf{F}}_p$ of  transcendence degree 2n
 and $\kappa$   the separable closure of $k$. Let $K$ be a
complete discretely valued field of characteristic 0 with residue field
$\kappa$.  Then  the Brauer $p$-dimension of $\kappa$ is 0 ([A], Ch.IV, \S 7, Theorem
18)  and Br$_p$dim$(K) \geq n$ (2.6).  Note that the $p$-rank of $\kappa$ is $2n$. 
Thus in the mixed characteristic case, the bound on the Brauer dimension
of the  residue field should be replaced by the  bound on the $p$-rank of
the residue field in order to get a good bound  on the Brauer dimension of a
complete discretely valued field. 
\end{example}

\section{The main theorem}

Let $R$ be an integral domain and $K$ its field of fractions.
Let $A$ be a central simple algebra over $K$. We say that $A$ is
{\it unramified  on} $R$ if there exists an Azumaya algebra ${\AA}$
over $R$ such that ${\AA} \otimes_R K$ is Brauer equivalent to
$A$.   If  $P$ is a prime ideal of $R$ and $A$ is unramified on $R_P$,
then we say that $A$ is {\it unramified at} $P$. If $\nu$ is a
discrete valuation of $K$ with $R$ as the valuation ring at $\nu$ and
$A$ is unramified on $R$, then we also say that $A$ is {\it unramified at}
$\nu$.

Let ${\XX}$ be a regular  integral scheme with function field $K$ and $A$ a
central simple algebra over $K$. Let $x \in {\XX}$ be a point. If
$A$ is unramified on the local ring ${\OO}_{{\XX}, x}$ at $x$,
then we say that $A$ is {\it unramified at} $x$.  If $A$ is not unramified
at $x$, then we say that $A$ is {\it ramified at} $x$.   
The {\it ramification divisor} of
$A$ on ${\XX}$ is the divisor $\sum x$,  where sum is taken over
the  codimension one points $x$ of ${\XX}$ with $A$ ramified at
$x$.  The {\it support} of the ramification divisor of $A$ is simply
the union of codimension  one points of ${\XX}$ where $A$ is ramified.

Let $T$ be a complete discrete valuation ring with field of fractions $K$ and 
$t  \in T$ a parameter. 
Let ${\XX}$ be an excellent regular, integral, 
proper scheme over Spec$(T)$ of dimension 2 with function field $F$  
and  reduced special fibre $Y$.  
For a closed point $P$ of ${\XX} $,  let $\OO_{\XX, P}$ denote the local ring at $P$,
$\hat{\OO}_{\XX, P}$  the completion of the regular local ring $\OO_{\XX, P}$ 
at its maximal ideal and $F_P$ the field of
fractions of $\hat{\OO}_{\XX, P}$.   For an open subset $U$  of an irreducible component  of $Y$, 
let $R_U$ be the ring consisting of elements in $F$ which
are regular on $U$.  Then $T \subset R_U$.   
Let $\hat{R}_U$ be the $(t)$-adic completion of $R_U$ and
$F_U$ the field of fractions of $\hat{R}_U$ (cf. [HHK1]). In this section we give
a bound for  the Brauer $p$-dimension of $F$ in terms of  the  $p$-rank of
the residue field  of $T$.    
 
We begin with the following results (3.1, 3.2, 3.3 and 3.4)  which
are  well-known and  we include them   for the sake of completeness. 
 
\begin{lemma}  Let $B$ be a regular local ring with field of
fractions $K$, residue field $\kappa$ and maximal ideal $m$.   
Let $n$ be a natural number  and $u \in B$   a  unit such that $[\kappa(
\sqrt[n]{\overline{u}}) :    \kappa] = n$. 
Then $B[\sqrt[n]{u}]$ is a regular local ring with
residue field $\kappa(\sqrt[n]{\overline{u}})$.
\end{lemma}

\begin{proof}  Since $[\kappa(\sqrt[n]{\overline{u}}) :
\kappa] = n$, 
$B[\sqrt[n]{u}]/mB[\sqrt[n]{u}] \simeq \kappa(\sqrt[n]{\overline{u}})$
is a field. Thus $m$ generates the maximal ideal of $B[\sqrt[n]{u}]$.
Since the dimension of $B[\sqrt[n]{u}]$ is equal to the dimension of
$B$,   $B[\sqrt[n]{u}]$ is a regular local ring.  
\end{proof}

\begin{lemma}    Let $B$ be a regular local ring with field of
fractions $K$, residue field $\kappa$ and maximal ideal $m$.    
Let $\pi \in m$ be a regular prime and $n$ a natural number. 
Then $B[\sqrt[n]{\pi}]$ is a regular local ring with
residue field $\kappa$.
\end{lemma}

\begin{proof}  Since $B$ is a regular and $\pi \in m$ is  a
regular prime, there  exist $\pi_2, \cdots, \pi_d \in m$ such that $m
= (\pi, \pi_2, \cdots  , \pi_d)$, where $d$ is the dimension of $B$.  
Let $\tilde{m} = (\sqrt[n]{\pi}, \pi_2, \cdots, \pi_d)  \\ \subset
B[\sqrt[n]{\pi}]$. Then $\tilde{m}$ is the maximal ideal of
$B[\sqrt[n]{\pi}]$ and $B[\sqrt[n]{\pi}]/\tilde{m} \simeq \kappa$. 
Since the dimension  of $B[\sqrt[n]{\pi}]$ is $n$, $B[\sqrt[n]{\pi}]$ is  a regular
local ring. 
\end{proof}

\begin{cor}   Let $B$ be a regular local ring with field of
fractions $K$, residue field $\kappa$ and maximal ideal $m$.  
Let $n_1, \cdots , n_r$ be natural numbers and $u_1, \cdots , u_r \in
B$  units such that
$[\kappa(\sqrt[n_1]{\overline{u}_1},  \cdots , \sqrt[n_r]{\overline{u}_r})
: \kappa] = n_1 \cdots n_r$.  Let  $\pi_1, \cdots , \pi_s \in m$ be a system of
regular parameters  in $B$ and  $d_1, \cdots , d_s $ be natural
numbers.  Then   $B[ \sqrt[n_1]{u_1}, \cdots ,
  \sqrt[n_r]{u_r}, \sqrt[d_1]{\pi_1}, \cdots , \sqrt[d_s]{\pi_s}]$ is a regular local
ring with  residue field \\ $\kappa(\sqrt[n_1]{\overline{u}_1}, \cdots ,
\sqrt[n_r]{\overline{u}_r})$.
\end{cor}

\begin{proof}  Proof follows by induction  using (3.1) and
(3.2). 
\end{proof}

\begin{lemma} (cf. [LPS], 2.4)  Let $R$ be a discrete valuation ring with 
field of fractions $K$. Let $\hat{R}$ be the completion of $R$ at the discrete
valuation and $\hat{K}$  the field of fractions of $\hat{R}$. 
Then  a central simple algebra $D$  over $K$ is unramified at $R$ if
and only if  $D \otimes_K \hat{K}$ is unramified at
$\hat{R}$. 
\end{lemma}

\begin{prop} Let $A$  be a  complete regular
local ring of dimension 2 with field of fractions $F$, residue field
$\kappa$ and maximal ideal $m = (\pi, \delta)$. 
Suppose that char$(F) = 0$ and char$(\kappa) = p > 0$ with
$p$-rank$(\kappa) = n$.  Let $a_1, \cdots , a_n \in \kappa$ be a
$p$-basis of $\kappa$ and  $u_1,  \cdots , u_n \in A$ be lifts of
$a_1, \cdots , a_n$.  Suppose that 
$F$ contains a primitive $p^{\rm th}$ root of unity. Let
$D$ be a central simple algebra over $F$ of period $p$.  Suppose that
$D$ is ramified on $A$ at most at $(\pi)$ and $(\delta)$.  Then $D
\otimes_F  F(\sqrt[p^2]{u_1}, \cdots ,\sqrt[p^2]{u_n}, \sqrt[p]{\pi},
\sqrt[p]{\delta})$ is a matrix algebra.  In particular, index$(D)$ divides $p^{2n+2}$. 
\end{prop}

\begin{proof}    Let $$E = F(\sqrt[p^2]{u_1}, \cdots ,\sqrt[p^2]{u_n}, \sqrt[p]{\pi},
\sqrt[p]{\delta})$$
and    
$$ B = A[\sqrt[p^2]{u_1}, \cdots ,\sqrt[p^2]{u_n}, \sqrt[p]{\pi},
\sqrt[p]{\delta}]  \subset E.$$ 
By (3.3), $ B $ is a complete regular local ring of
dimension  2 with field of fractions $E$ and residue field
$\kappa(\sqrt[p^2]{a_1}, \cdots \sqrt[p^2]{a_n})$.   

 We first show that $D \otimes_F E$ is  unramified on
 $B$. Since $ B $ is a regular local ring of dimension 2, it is enough
 to show that $D \otimes_FE$ is unramified at every height one prime
 ideal  of $B$ ([AG], 7.4). Let $Q$ be a height one prime ideal of
 $B$ and $P = Q \cap A$. Since $ B$ is integral over
 $A$, the height of $P$ is 1.  If $P \neq (\pi) $ or $(\delta)$, then $D$
 is unramified at $P$ and hence $D \otimes_F E$ is unramified at
 $Q$. Suppose that $P = (\pi)$. Then $Q = (\sqrt[p]{\pi})$. 
 
 Suppose that char$(A/P) \neq p$.  Since $E/F$ is ramified at $P$ and
 char$(\kappa(P)) \neq p$,  $D \otimes _FE$ is
 unramified at $Q$.   Suppose that char$(A/P) = p$.  Since $A$ is
 complete regular local ring with  maximal ideal
 $m = (\pi, \delta)$,  $A/(\pi)$ is a complete discrete valuation ring with residue field 
 $\kappa$ and char$(A/P) = $ char$(\kappa)= p$. In particular, $A/(\pi)  \simeq
 \kappa[[\overline{\delta}]]$, where $\overline{\delta}$    is the
 image of $\delta$ in $A/(\pi)$.  Let $\kappa(P)$ be the field of fractions of $A/P$.
 Then $\kappa(P) \simeq \kappa((\overline{\delta}))$.   Since $a_1, \cdots , a_n$ is a 
 $p$-basis of $\kappa$ and $u_1, \cdots, u_n \in A$ are  lifts of   $a_1, \cdots , a_n$, 
 the images of $u_1, \cdots, u_n,  \delta$ in $\kappa(P)$ is a $p$-basis of $\kappa(P)$.  
 Let $F_P$ be the completion of $F$ at $P$ and $E_Q$  the completion of $E$ at $Q$. 
 Since  $E_Q \simeq F_P(\sqrt[p^2]{u_1}, \cdots ,\sqrt[p^2]{u_n}, \sqrt[p]{\delta},\sqrt[p]{\pi})$ and the 
 residue field of $F_P$ is $\kappa(P)$,  by  (2.4),   $D \otimes _FE_Q$ is split and hence unramified.  
Thus,  by (3.4),  $D \otimes_F E $ is unramified at  $Q$.

By ([AG], 7.4), there exists an Azumaya $B$-algebra  ${\DD}$   such
that ${\DD}  \otimes_B E \simeq D \otimes _F E$.  Since $D \otimes E_Q$ is
split and $\hat{B}_Q$ is a discrete valuation ring, ${\DD} \otimes_B \hat{B}_Q$ 
is zero in the Br$(\hat{B}_Q)$ ([AG], 7.2). In particular the  image
${\DD} \otimes_B \kappa(Q)$ of ${\DD} \otimes_B \hat{B}_Q$  
in Br$(\kappa(Q))$ is zero.  Since $\kappa(Q)$ is the field of
fractions of  regular local ring $B/Q$,  by ([AG], 7.2), 
${\DD} \otimes _B B/Q$ is zero in Br$(B/Q)$. Hence ${\DD}\otimes
_B B/\tilde{m}$ is zero in Br$(B/\tilde{m})$, where  
$\tilde{m}$ is the maximal ideal of $B$.  Since $B$ is a complete regular local ring, 
${\DD} = 0 \in$ Br$(B)$ ([C], [KOS]). In particular ${\DD}
\otimes _B E \simeq D \otimes _F E$ is zero and  
 index$(D)$ divides $[E : F] = p^{2n+2}$.   
\end{proof}

\begin{theorem} Let $K$ be a complete discretely valued field
  with   residue field $\kappa$.   Suppose that char$(K) = 0$, 
char$(\kappa) = p > 0$ and  $p$-rank$(\kappa) = n$.  Let $F$ be a finitely generated 
field extension of $K$ of transcendence degree 1  
and $D$ a central simple algebra over $F$ of period $p$. 
Then  ind$(D)$ divides $p^{2n + 2}$. 
\end{theorem}

\begin{proof}   As in the proof of (2.4), we assume without loss of
  generality that    $F$  contains a primitive $p^{\rm th}$  root of unity.
Let $K'$ be a finite extension of $K$. Then $K'$ is also a complete discretely valued
field with the $p$-rank of the residue field is  $n$.  Thus, replacing 
$K$ by a finite extension of $K$, we assume that  $F$ is  the 
function field of a geometrically integral  smooth  projective curve $X$ over
$K$.    

We choose a proper regular model ${\XX}$ of $F$ over $T$ such that
the support of the ramification divisor of $D$ and the components of
the reduced special fibre are a union of regular curves with normal
crossings on ${\XX}$.  Let $Y$ be the special fibre of ${\XX}$.

Let $\eta$ be a generic point of an irreducible component of $Y$
and $F_\eta$ the completion of $F$ at the discrete valuation given by
$\eta$. Then the residue field $\kappa(\eta)$ of $F_\eta$ is function
field of transcendence degree one over $\kappa$. Since $[\kappa :
\kappa^p] = p^n$, we have $[\kappa(\eta) : \kappa(\eta)^p] = p^{n+1}$.
By (2.4),  ind$(D \otimes_F  F_\eta)$ divides $p^{2n+2}$. By ([HHK2],
5.8 and [KMRT],  1.17),
there exists an irreducible open set $U_\eta$ of $Y$ containing $\eta$ such
that ind$(D \otimes_F F_{U_\eta}) = $ ind$(D \otimes _F F_\eta)$. In
particular ind$(D \otimes_F F_{U_\eta})$ divides $p^{2n+2}$.   

Let $S_0$ be a finite set of closed points of ${\XX}$ containing 
all the points of   intersection of the components of  $Y$ and the
support of the ramification divisor of $D$.  Let $S$ be a finite set
of closed points of ${\XX}$ containing $S_0$ and $ Y \setminus (\cup~
U_\eta)$, where $\eta$ varies over generic points of $Y$. Then, by
([HHK1], 5.1), 
 $$ {\rm ind}(D) = l.c.m \{ {\rm ind}(D \otimes F_\zeta \},$$
where $\zeta$ running over $S$ and irreducible components of $Y
\setminus S$. 

Suppose $\zeta = U$ for some irreducible component $U$ of $Y \setminus
S$. Let $\eta$ be the generic point of $U$.  Then $U \subset
U_\eta$ and $R_{U_\eta} \subset R_U$. Since $F_{U_\eta} \subset F_U$,
ind$(D \otimes_F F_U)$ divides  $p^{2d+2}$. 

Suppose $\zeta = P \in S$.  Let $A_P$ be the regular local ring at
$P$. Then, by the choice of ${\XX}$, the maximal ideal $m_P$ of
$A_P$ is generated by $\pi$ and $\delta$ such that $A$ is ramified on
$A_P$ at most  possibly at $(\pi)$ and $(\delta)$. Since the residue
field $\kappa(P)$ at $P$ is a finite extension of $\kappa$, we have
$p$-rank$(\kappa(P)) = p$-rank$(\kappa) =  p^n$. Thus, by (3.5), ind$(D \otimes_F
F_P)$ divides $p ^{2n+2}$.  Hence ind$(D)$  divides
$p^{2n+2}$.   
\end{proof}

\begin{cor} Let  $F$ and $n$ be as in (3.6). Then Br$_p$dim$(F) \leq 2n+2$. 
\end{cor}

\begin{proof} Let $F'$ be a finite extension of $F$ and $D$ a central simple algebra
of period $p$.  Since the transcendence degree of $F'$ over $K$ is 1,
by (3.6), ind$(D)$  divides $p^{2n+2}$.
Corollary follows from (1.6). 
\end{proof}

\begin{cor}  Let $K$ be a complete discretely valued field with
residue field $\kappa$. Suppose that $\kappa$ is finitely generated
field of transcendence   degree  $n$ over a perfect field of
characteristic $p>0$.  If $F$ is a function field of a curve over $K$,
then   the   Brauer $p$-dimension of $F$
is at most $2n+2$.   
\end{cor}

\begin{proof}  Since $\kappa$ is a finitely generated field of
transcendence degree $d$ over a perfect field, we have $[\kappa :
\kappa^p] =  p^n$ ([B], A.V.135, Corollary 3). 
Hence the result follows from (3.6).  
\end{proof}
 
Let $K$ be a $p$-adic field and $F$ the function field of curve over
$K$. Let $A$ be a central simple algebra over $F$ . If the  period of $A$ is 
coprime to
$p$, then a theorem of Saltman ([S1]) asserts that ind$(A)$ divides
per$(A)^2$.  If the period of $A$ is a power of $p$, then it is proved in 
([LPS]) that the ind$(A)$ divides per$(A)^3$.  We have the following

\begin{cor} Let $F$ be the function field of a curve
over a $p$-adic field $K$. Then for every central simple algebra over
$F$, the index divides  the square of the period. 
\end{cor}

\begin{proof}  Let $A$ be a central simple algebra over $F$ of
period a power of $p$. Since the residue field $\kappa$ of $K$ is a
finite field,  $[\kappa : \kappa^p] = 1$. Thus, by (3.6), ind$(A)$
divides per$(A)^2$.  
 \end{proof}

\section{ $u$-invariant}

Let $K$ be a complete discretely valued field with residue field
$\kappa$ and $F$ the  function field of a curve over $K$. 
In this section we compute   the  $u$-invariant of $F$
when $\kappa$ is a perfect field of characteristic 2 and char$(K) =
0$.  

For any field $L$ of characteristic not equal to 2, let $W(L)$ be the
Witt ring of quadratic forms over $L$ and $I^n(L)$ be the $n^{\rm th}$
power of the fundamental ideal $I(L)$ of $W(L)$.   

Let $R$ be an integral domain  with field of fractions $F$. 
A quadratic form $q$ over $R$ is {\it non-singular} if the associated 
quadric is smooth over $R$. 
We say that  
a quadratic form $q$ over $F$ is {\it defined over} $R$ if there exists a non-singular 
quadratic form $q'$ over $R$ such that $q' \otimes_R F \simeq q$.  

In the rest of this section, until (4.7),  $A$ denotes  a complete regular local ring of
dimension two with field of fractions $F$ and residue field $\kappa$.
Suppose that char$(F) = 0 $, char$(\kappa) = 2$ and $\kappa$ is a perfect
field.  Suppose that the maximal ideal $m = (\pi, \delta)$ and $2 =
u_0\pi^i\delta^j$ for some $u_0 \in A^*$  and $i, j \geq 0$.  

\begin{lemma}
Let $A$, $F$, $\kappa$, $m  = (\pi, \delta)$ as above.
Let $\alpha \in H^2(F, \mu_2)$. If $\alpha $ is unramified on $A$ except possibly 
at $(\pi)$ and $(\delta)$. Then $\alpha = (uc, \pi) + (vc\pi^\epsilon, \delta)$ for some 
units $u, v \in A$, $c \in A$ not divisible by $\pi$, $\delta$ and
$\epsilon = 0$ or 1. 
\end{lemma}
 
 \begin{proof}
 Since $\alpha$ is unramified except at $(\pi)$ and $(\delta)$ and $
\kappa$ is perfect, by (3.5), $\alpha \otimes F(\sqrt{\pi},
\sqrt{\delta})$ is zero.  In particular,  by a theorem of Albert,
$\alpha = (a, \pi) + (b, \delta)$ for   some $a, b \in F^*$. Without
loss of generality we assume that $a, b  \in A$ and square free. 
Since $(-d, d) = 0$ for any $d \in F^*$, we assume that  $\pi$ does
not divide $a$ and $\delta$ 
does not divide $b$. Since $A$ is a   regular local ring,  it is a
unique factorisation domain ([AB]). We write  $a = ca_1\delta^{\epsilon_1}$ and $b =
cb_1\pi^{\epsilon_2}$ with $c, a_1, b_1 \in A$  square free,  $a_1$
and $b_1$ are coprime, $\pi$ and $\delta$ do
not divide $ca_1b_1$ and $0 \leq \epsilon_1, \epsilon_2 \leq 1$.

Let $\theta$ be a prime in $A$ which divides $a_1$. Write $a_1 = \theta
a_2$.  Then $\theta$ does  not
divide $cb_1\pi\delta$.  In particular, the characteristic of the
residue field $\kappa(\theta)$ at $\theta$ is not equal to 2 and
$\alpha$ is unramified at $\theta$.  Since the residue of $\alpha$ 
 at $\theta$ is the image $\overline{\pi}$ of $\pi$ in
$\kappa(\theta)/\kappa(\theta)^{*2}$,  $\overline{\pi}$ is a square in
$\kappa(\theta)$. Let $L = F[\sqrt{\pi}]$ and $B =
A[\sqrt{\pi}]$. Then $B$ is a regular local ring of dimension 2
(cf. (3.2)) and hence a unique factorisation domain ([AB]). 
Since $\overline{\pi}$ is a square in $\kappa(\theta)$
and  char$(\kappa(\theta)) \neq 2$, we have $\theta B = Q_1Q_2$  with 
$Q_1$ and $Q_2$ two distinct  prime ideals of $B$.   In particular $N_{L/F}(Q_1)
= \theta A$.  
Since  $B$ is a unique factorisation domain,  $Q_1 = (\eta)$ for some
$\eta \in B$ and hence there exists a unit $u \in A$ such that   
$N_{L/F}(\eta) = u\theta$. We have 
$$
\begin{array}{rcl}
(a, \pi)  & = & (au\theta, \pi)  \\
& = &  (ca_1\delta^{\epsilon_1}u\theta, \pi) \\
& = & (c\theta a_2\delta^{\epsilon_1}u\theta, \pi) \\
& = & (ca_2\delta^{\epsilon_1}u, \pi).
\end{array}
$$
Thus by induction on the number of primes dividing $a_1$, we conclude
that $(a, \pi) = (uc\delta^{\epsilon_1}, \pi)$ for some unit $u \in
A$.
Similarly $(b, \delta) = (vc\pi^{\epsilon_2}, \delta)$ for some unit
$v \in A$. Thus we have $\alpha =  (uc\delta^{\epsilon_1}, \pi) + 
(vc\pi^{\epsilon_2}, \delta)$. Suppose that $\epsilon_1 = 1$. Then 
$$
\begin{array}{rcl}
 \alpha & = &  (uc\delta, \pi) +  (vc\pi^{\epsilon_2}, \delta) \\
& = & (uc, \pi) + (\delta, \pi) +  (vc\pi^{\epsilon_2}, \delta) \\
& = & (uc, \pi) + (vc\pi^{\epsilon_2 + 1}, \delta) \\
& = & (uc, \pi) + (vc\pi^{\epsilon}, \delta),
 \end{array}
$$
where $\epsilon = \epsilon_2 + 1$ (mod 2).
\end{proof}

For any field $K$ and  $i
\geq 1$,  let $H^i_2(F)$ be the Kato cohomology groups ([K2], \S 0).
If char$(K) \neq 2$, we have $H^i_2(F) =   H^i(F, \mu_2)$.  If
char$(K) = 2$, we have $H^2_2(K) = \Brt(K)$ and $H^1_2(K) 
=  H^1( K, \mathbf{Z}/2\mathbf{Z})$.  For $a \in K^*$, let $[a) \in H^1_2(K)$ be the
element defined by $ K[X]/
 (X^2 + X + a)$.  Note that   $ [a) $ is $ K \times K$ or  a
 separable extension of $K$ of degree  2.  Let $b \in K$. Let $[a)
 \cdot (b)$ be the quaternion  algebra over $K$ generated by $i$ and
 $j$ with $i^2 + i + a = 0$, $j^2 = b$ and $ji = -(1+ i) j$. 

Let $A$, $F$, $\kappa$ be as above.  
Let $\theta \in A$ be a prime.
Suppose
$\theta$ does not divide $2 = u_0\pi^i\delta^j$. Then the characteristic of the
residue field $\kappa(\theta)$ at $\theta$ is 0.  Suppose $\theta$
divides 2. Then $(\theta) = (\pi)$ or $(\theta) = (\delta)$ and $A/(\theta)$ is a
complete discrete valuation ring. 
Since the residue field $\kappa$ of $A$ is a perfect field, we have 
$[\kappa(\theta) : \kappa(\theta)^2] = 2$.
By ([K2], \S 1), we
have residue homomorphisms $\partial_\theta : H^3(F, \mu_2) \to
H^2_2(\kappa(\theta)) \simeq  \Brt(\kappa(\theta))$ and $\partial :
H^2_2(\kappa(\theta)) \to H^1_2(\kappa)$.   

\begin{lemma}
(cf. [Su], 1.1) Let $A$, $F$, $\kappa$, $m = (\pi, \delta)$ be as above. Then, for any 
unit $u \in A^*$,    $\partial_\pi( [u) \cdot (\delta) \cdot (\pi ))=
[\overline{u}) \cdot (\overline{\delta})$ and  $\partial ([\overline{u}) \cdot 
(\overline{\delta}) ) = [\overline{\overline{u}})$, where for any $a \in A$, $\overline{a}$ denotes the image 
modulo $\pi$ and  $\overline{\overline{a}}$ denotes the image modulo $m$. 
\end{lemma}

\begin{proof}  Suppose that $[\overline{\overline{u}}) $ is trivial in $H^1_2(\kappa)$.
Since $u$ is a unit  in $A$ and $A$ is complete, $[u) $ is trivial in $H^1(F, \mu_2)$.
In particular $[u) \cdot (\pi) \cdot (\delta)$ and $[\overline{u}) \cdot (\overline{\delta})$ are trivial. 

Suppose that $[\overline{\overline{u}})$ is non-trivial.  Let $\kappa' = \kappa[X]/(X^2 + x + 
\overline{\overline{u}})$. Then $\kappa'$ is a separable quadratic extension of $\kappa$
and $[\overline{\overline{u}})$ is the only  non-trivial element of the 
kernel of the restriction homomorphism from $H^1_2(\kappa)$ to $H^1_2(\kappa')$. 

Let $\kappa(\pi)' = \kappa(\pi)[X]/(X^2 + X + \overline{u})$. Then $\kappa(\pi)'$ is 
a complete discretely valued field with residue field $\kappa'$ and $\overline{\delta}$ 
as a parameter. Thus $\kappa(\pi)'/\kappa(\pi)$ is unramified and  $\overline{\delta}  \in \kappa(\pi)$
 is not a norm from $\kappa(\pi)'$ and hence
$[\overline{u}) \cdot (\overline{\delta})$ is non-trivial.  
Since $\partial$ is an isomorphism ([K2], Lemma 1.4(3)), 
 $\partial ([\overline{u}) \cdot (\overline{\delta}))$ is non-trivial in $H^1_2(\kappa)$. 
Since $[\overline{u}) \cdot 
(\overline{\delta})$ is trivial over $\kappa(\pi)'$, by the functoriality of $\partial$, 
the image of $\partial([\overline{u}) \cdot (\overline{\delta}))$ in $H^1_2(\kappa')$ is trivial. 
Since the only non-trivial  element of the 
kernel of the restriction homomorphism from $H^1_2(\kappa)$ to $H^1_2(\kappa')$ is $[\overline{\overline{u}})$,
$\partial([\overline{u}) \cdot (\overline{\delta}) ) =  [\overline{\overline{u}})$. 
  
 Let $F_\pi$ be  the completion  of $F$ at $\pi$.
  Since  $u$  and $\delta$ are units at $\pi$, $[u) \cdot (\delta)$ is a quaternion algebra defined over 
$A_{(\pi)}$.  If $\pi$ is a reduced norm from $[u) \cdot (\delta)$ over $F_\pi$,  $[\overline{u}) \cdot (\overline{\delta})$ 
is a split algebra over $\kappa(\pi)$,  contradicting the non-triviality of $[\overline{u}) 
\cdot (\overline{\delta})$ in $H^2_2(\kappa(\pi))$.
Hence  $\pi$  is not a reduced norm  form of the quaternion algebra $([u) \cdot (\delta)) \otimes_F F_\pi$
and  $[u) \cdot (\delta) \cdot (\pi)$ is non-trivial in $H^3(F_\pi, \mu_2)$.
 Let $L = F [X]/(X^2 + X + u)$.  Let $B$ be the integral closure of $A$ in $L$. Then $B$ is a complete 
 regular local ring with maximal ideal $(\pi, \delta)$ and residue field $\kappa'$. 
   Since the image of $[u) \cdot (\delta) \cdot (\pi)$ 
in $H^3(L, \mu_2)$ is zero,  by the functoriality of the residue 
homomorphisms,  the image of  $\partial(\partial_\pi([u) \cdot (\delta) \cdot (\pi )))$ in $H^1_2(\kappa')$
is zero.  Since $[u) \cdot (\delta) \cdot (\pi)$ is non-trivial in $H^3(F_\pi, \mu_2)$ and 
$\partial_\pi : H^3(F_\pi,  \mu_2) \to H^2_2(\kappa(\pi))$ and $\partial : H^2_2(\kappa(\pi)) \to H^1_2(\kappa)$
are isomorphisms ( [K2], Lemma 1.4(3)),  $\partial(\partial_\pi([u) \cdot (\delta) \cdot (\pi ))$ is non-trivial
and hence equal to $[\overline{\overline{u}})$.  
 Since $\partial ([\overline{u}) \cdot (\overline{\delta}) = [\overline{\overline{u}}) $ and
 $\partial$ is an isomorphism, we have  $\partial_\pi([u) \cdot (\delta) \cdot (\pi ) ) =
[\overline{u}) \cdot (\overline{\delta})$. 
\end{proof}

The following is a result of Kato ([K2], 1.7)

 \begin{prop} 
 Let $A$, $F$ and $\kappa$ be as
above. Then   
$$
H^3(F, \mu_2) \buildrel{\oplus\partial_\gamma}\over{\to}
{\displaystyle{\oplus_{\gamma \in {\rm Spec}(A) ^{(1)}}}}
H^2_2(\kappa(\gamma))  \buildrel{\sum\partial_\gamma}\over{\to}
H^1_2(\kappa) 
$$ 
is a complex. 
\end{prop}

We define $H^3_{nr}(F/A, \mu_2)$ to be the kernel of residue
homomorphism 
$$ H^3(F, \mu_2)
\buildrel{\oplus\partial_\gamma}\over{\to}   {\displaystyle{\oplus_{\gamma \in Spec(A)^{(1)}}}}
H^2_2(\kappa(\gamma))).$$

\begin{prop}Let $A$, $F$ and $\kappa$ be as above.  
Then $H^3_{nr}(F/A, \mu_2) = 0$. 
\end{prop}
 
 \begin{proof}
  Since cd$_2(F) \leq 3$ ([GO]),   $I^4(F) = 0$
([AEJ], Cor.2. p.653) and $e_3 : I^3(F) \to
H^3(F, \mu_2)$ is an isomorphism ([AEJ], Thm. 2. p.653).  Let $\zeta \in H^3_{nr}(F/A,
\mu_2)$. Suppose that $\zeta \neq 0$.  
Let $q $ be an anisotropic quadratic form over $F$ such that
$q \in I^3(F)$ and   $e_3(q) = \zeta$. 
Let $\theta \in A$ be a prime and $F_\theta$ be the completion of $F$
at $\theta$.  Since $A/(\theta)$ is a complete local ring of dimension
one with residue field perfect of characteristic 2,
$H^3_2(\kappa(\theta)) = 0$ ([GO]). Suppose that  char$(\kappa(\theta)) \neq 2$.
Since $H^3(\kappa(\theta), \mu_2) = H^3_2(\kappa(\theta)) = 0$,  
$\partial_\theta : H^3(F_\theta, \mu_2) \to H^2_2(\kappa(\theta)) $
is an isomorphism. Suppose that  char$(\kappa(\theta)) = 2$. Since the $2$-rank of 
$\kappa(\theta)$ is 1, by ([K2], Lemma 1.4(3)),  $\partial :
H^3(F_\theta,  \mu_2) \to H^2_2(\kappa(\theta))$ is an
isomorphism. Since $\zeta \in H^3_{nr}(F/A, \mu_2)$, the image of
$\zeta$ in $H^3(F_\theta, \mu_2)$ is zero.  In particular,  $q$
is hyperbolic over $F_\theta$. 
Thus $q$  comes from a non-singular  quadratic form 
over the localisation $A_{(\theta)}$ of $A$ at the prime ideal
$(\theta)$ (cf. [O], Thm,8). Since $A$ is a two
dimensional regular ring,   there exists a non-singular
quadratic forn $q'$ over $A$ such that $q' \otimes _A F \simeq
q$ ([CTS], Cor.2.5, cf. [APS], 4.2 ).  

Since $q \in I^3(F)$ and $q$ is anisotropic, 
the rank of $q$, and hence the rank of $q'$, is
at least 8. Since $\kappa$ is a perfect field, $q' \otimes_A \kappa$
is isotropic ([MMW], Corollary 1). Since $A$ is a complete regular
local ring and $q'$ is a non-singular quadratic form over $A$ with   
$q' \otimes_A \kappa$ is isotropic, $q'$ is isotropic ([Gr], Theorem 18.5.17).
Thus $q$ is isotropic, leading to a contradiction.  
\end{proof}

\begin{lemma}Let $A$, $F$, $\kappa$ and  $m = (\pi, \delta)$ be as
be above.  Let $\zeta \in H^3(F,
\mu_2)$. Suppose that  $\zeta$ is ramified at most  along  $(\pi)$ and
$(\delta)$. Then $\zeta = [u) \cdot (\pi) \cdot (\delta)$ for some
unit $u$  in $A$.  
\end{lemma}

\begin{proof} 
 Let $\alpha  = \partial_\pi(\zeta) \in H^2_2(\kappa(\pi))$ and $\beta
= \partial_\delta(\zeta) \in H^2_2(\kappa(\delta))$. Then, by
(4.3), $\partial(\alpha)   = \partial(\beta) \in
H^1_2(\kappa)$. Let $a \in \kappa^*$ be such that    $[a)
= \partial(\alpha)   = \partial(\beta) \in 
H^1_2(\kappa)$. Let $u \in A^*$ be a lift of $a$. 
Since $\partial(\alpha) = [a) = \partial([u) \cdot (\overline{\delta}))$ (cf. 4.2)
and $\partial$ is an isomorphism, we have  $\alpha = [\overline{u}) \cdot ( \overline{\delta})$.
 and $\beta = [\overline{u}) \cdot  ( \overline{\pi})$. Let $\zeta' = [u) \cdot (\pi) \cdot
(\delta) \in H^3(F, \mu_2)$.  Then $\zeta'$ is unramified on $A$
except at $\pi$ and $\delta$. By (4.2),  $\partial_\pi(\zeta')
= \partial_\pi(\zeta)$ and $\partial_\delta(\zeta')
= \partial_\delta(\zeta)$. Since $\zeta$ is unramified on $A$ except
at $\pi$ and $\delta$, $\zeta - \zeta' \in H^3_{nr}(F, \mu_2)$. Since 
$H^3_{nr}(F,\mu_2) = 0$ by (4.4), we have 
$\zeta = \zeta' = [u)\cdot (\pi) \cdot (\delta)$.
\end{proof}

\begin{prop}  Let $A$, $F$, $\kappa$ and $m = (\pi, \delta)$
be as above.  Let $q = <a_1, \cdots , a_9>$ be a quadratic
form over $F$ of rank 9 with  only prime factors  of  $a_1a_2 \cdots
a_9$  are  at most $\pi$ and $\delta$. Then $q$ is isotropic. 
\end{prop}

\begin{proof}  
Let $c(q) \in H^2(F,\mu_2)$ be the Clifford invariant of $q$. 
Since  the prime factors of $ a_1a_2 \cdots a_9$
are at most $\pi$ and
$\delta$, $c(q)$ is unramified on $A$ except possibly at $(\pi)$ and
$(\delta)$.   By  (4.1), we have 
$c(q) = (uc, \pi) +  (vc\pi^{\epsilon}, \delta)$ for some units $u, v
\in A$, $c \in A$ not divisible by $\pi$ and $\delta$, and $\epsilon =
0$ or 1. Let $q_1 = <1, uc\pi, -\pi,  -uc\delta,
uv\pi^\epsilon\delta>$. Since $-ucq_1$ is a rank five subform of the
Albert form associated to $c(q) = (uc, \pi) +
(vc\pi^{\epsilon},\delta)$, $c(q_1) = c(q)$ (cf. [L], p. 118).  Since $q$ is isotropic
if and only if $\lambda q$ is isotropic for any $\lambda \in F^*$, by
scaling $q$ we assue that $d(q) = d(q_1)$.  We note that we only 
need to scale
by $\lambda \in A$ with prime factors at most $\pi$ and
$\delta$. Hence, after scaling, we still have $q = <a_1, \cdots,
a_9>$ with only prime factors of $a_1a_2 \cdots a_9$  at most
$\pi$ and $\delta$.    Since the dimension of $q$ is odd, we have
$c(\lambda q) = c(q)$. Thus, after scaling, we have $c(q) = c(q_1)$
and $d(q) = d(q_1)$. Since the rank of $q \perp -q_1$ is 14, it
follows that $q - q_1 \in I^3(F)$ ([M]).  

Let $\zeta = e_3(q - q_1) \in H^3(F, \mu_2)$. 
Let $\theta \in A$ be a prime. Suppose that $\theta$ does not divide
$\pi\delta$.  Then char$(\kappa(\theta))$ is 0. Hence we have 
the second residue homomorphism   $\partial^2_\theta : W(F) \to
W(\kappa(\theta))$ with $\partial^2_\theta (I^3(F)) \subset
I^2(\kappa(\theta))$. Since $q = <a_1, \cdots , a_9>$ with
$ a_1a_2 \cdots a_9$ having only  
$\pi$ and $\delta$ as possible prime factors  and $\theta$ does not divide $\pi\delta$,
$\partial^2_\theta(q) = 0$.  Since $q_1 = <1, uc\pi, -\pi,  -uc\delta,
uv\pi^\epsilon\delta>$, the rank of $\partial^2_\theta(q_1)$ is at
most two. Since $\partial^2_\theta(q - q_1) \in I^2(\kappa(\theta))$
and is of rank at most 2,  $\partial^2_\theta(q  - q_1) = 0$.
In particular $q - q_1$ is unramified at $\theta$ and hence $\zeta =
e_3(q - q_1)$ is unramified at $\theta$. Thus,  by  (4.5), we have 
 $\zeta  = [w) \cdot (\pi) \cdot (\delta)$ for some unit $w \in A$.
Since $[w) \cdot (w') $ is unramified on $A$ for any unit $w' \in A$,
we have $[w) \cdot (w') = 0$. In particular, we have 
$\zeta = [w) \cdot (\pi) \cdot (w'\delta)$ for any unit $w'$ in $A$. 

Suppose that $\epsilon = 0$.  Since $uv$ is a unit, we have 
 $\zeta = [w) \cdot (\pi) \cdot (-uv\pi^\epsilon\delta)$. 
Suppose that $\epsilon = 1$.  Since $\zeta =  [w) \cdot (\pi) \cdot
(-uv\delta)$ and  $(\pi) \cdot (-\pi) = 1$, we have $\zeta =  
[w) \cdot (\pi) \cdot (-uv\pi \delta)$. Thus in either case, 
we have $\zeta = e_3(q - q_1) = [w) \cdot (\pi) \cdot
(-uv\pi^\epsilon\delta)$.

Since char$(F)  = 0$, we have $[w) = (w')$ for some unit $w' \in A$.  
Let $q_2 = <1, -w'><1, -\pi><1, uv\pi^\epsilon\delta> \in I^3(F)$.
Then $e_3(-q_2) = e_3(q_2) = (w') \cdot (\pi) \cdot (
-uv\pi^\epsilon\delta) = [w) \cdot (\pi) \cdot (-uv\pi^\epsilon\delta)
= e_3(q - q_1)$.  
Since $H^4(F, \mu_2) = 0$ ([AEJ], Cor.2. p.653), we have $I^4(F, \mu_2) = 0$ and
$e_3$ is an isomorphism  ([AEJ], Thm. 2. p.653).
Hence  
$$
q - q_1 = -<1, -w'><1, -\pi><1, uv\pi^\epsilon\delta>.
$$
In  particular,
$$
q = q_1 - <1, -w'><1, -\pi><1, uv\pi^\epsilon\delta>.
$$
Since $<1, -\pi, uv\pi^\epsilon\delta>$ is a subform of both $q_1$ and
$<1, -w'><1, -\pi><1, uv\pi^\epsilon\delta>$, the anisotropic rank of 
$q_1 - <1, -w'><1, -\pi><1, uv\pi^\epsilon\delta>$ is at most 7.
Since the rank of $q$ is 9, $q$ is isotropic. 
\end{proof}

\begin{theorem}  Let $K$ be a complete discretely valued field with
  residue field $\kappa$ and $F$ a function field of a curve over
  $K$. If char$(K) =  0$ and $\kappa$ is a perfect field of
  characteristic 2, then $u(F)  \leq 8$.
\end{theorem}

\begin{proof}  Let $q = <a_1, \cdots , a_9>$ be a quadratic form over $F$ rank 9.
 Let ${\XX}$ be a regular proper   scheme over the valuation ring of
 $K$ with function field $F$ and the  support of the principle divisor  
$(2a_1 \cdots a_9)$ on ${\XX}$ 
 is a union of regular curves with
normal crossings. Let $C_1, \cdots , C_r$ be the irreducible
components of the special fibre of $\XX$ and let $\nu_1, \cdots,
\nu_r$ be the corresponding discrete valuations on $F$.  Let
$F_{\nu_i}$ be the  completion  $F$ at $\nu_i$ and the residue field
$\kappa(\nu_i)$.  Then  char$(\kappa(\nu_i)) = 2$ and
$2$-rank$(\kappa(\nu_i)) = 1$.   Hence, by ([MMW], Corollar 1),
$u(\kappa(\nu_i)) \leq 4$ and by ([Sp]), $u(F_{\nu_i}) \leq 8$. In
particular  $q$ is isotropic over $F_{\nu_i}$. 
 By ([HHK2], 5.8),  there exists an affine open subset
$U_i$ of $C_i$ such that $U_i$ does not intersect $C_j$ for $j \neq i$
and $q$ is isotropic over $F_{U_i}$.  

Let $\PP$ be a finite set of closed points of $\XX$ containing all
those points which are not in $U_i$ for any $i$.  Let $P \in \PP$.
Then $\hat{A}_P$ is a complete two dimensional local ring with residue
field perfect of characteristic 2.  By the choice of $\XX$ and  (4.6), $q$ is 
isotropic over $F_P$. 
By ([HHK1], 4.2), $q$ is isotropic over $F$ and $u(F) \leq 8$. 
\end{proof}

\begin{cor}([Le]) Let $K$ be a  2-adic field and $F$ the function field of a curve
over $K$. Then $u(K) = 8$.

\end{cor}

\section*{References}

\begin{enumerate}

\item[{[A]}] A. A. Albert,  \textit{Structure of Algebras},  
Amer. Math. Soc. Colloq. Publ., Vol. 24, Amer. Math. Soc., Providence,
RI,  1961, revised printing.

\item[{[AEJ]}]  J.K. Arason, R.Elman and B.Jacob, \textit{Fields of
    cohomological 2-dimensions three}, Math. Ann. {\bf 274} (1986),
  649-657.   

\item[{[ABGV]}] Asher Auel, Eric Brussel, Skip Garibaldi, Uzi Vishne,
  \textit{Open problems on central simple algebras}, Transformation
  Groups {\bf 16} (2011), 219-264.

\item[{[APS]}] Asher Auel, R. Parimala and V. Suresh, \textit{Quadric surface bundles over surfaces}, 
preprint 2012. 

\item[{[AB]}]  Auslander  Maurice and Buchsbaum, D.A, 
\textit{Unique factorization in regular local rings}, 
Proc. Nat. Acad. Sci. U.S.A. {\bf 45} (1959) 733-734. 

\item[{[AG]}]  Auslander  Maurice and  Goldman Oscar, 
\textit{The Brauer group of a commutative ring},
Trans. Amer. Math. Soc. {\bf 97} (1960) 367-409.
 
\item[{[B]}] N. Bourbaki,  \textit{Algebra II}, Springer-Verlag, New York, 1988.

\item[{[C]}]  M. Cipolla,  \textit{Remarks on the lifting of
    algebras over henselian pairs}, 
Mathematische Zeitschrift, {\bf 152} (1977), 253--257.

\item[{[CT]}]  J.-L. Colliot-Th\'el\`ene, \textit{ Cohomologie
    galoisienne des corps valu\'es discrets  henseliens, d'apr\`es
    K. Kato et S. Bloch},  Algebraic K-theory and its applications
  (Trieste, 1997), 120–163, World Sci. Publ.,  River Edge, NJ, 1999.

\item[{[CTS]}]    J.-L. Colliot-Th\'el\`ene et J.-J.  Sansuc,  
\textit{Fibr\'es quadratiques et composantes connexes r\'eelles},
Math. Ann.  {\bf 244}  (1979),   105-134. 

\item[{[GO]}] O. Gabber and F. Orgogozo, \textit{Sur la $p$-dimension
    des corps}, Invent. Math. {\bf 174} (2008), 47-80.

\item[{[GS]}] Philippe Gille and Tamm\'as Szamuely,
  \textit{Central simple algebras and Galois cohomology},  Cambridge
  Studies  in Advanced Mathematics, vol.101, Cambridge University
  Press, Cambridge, 2006.

\item[{[Gr]}] A. Grothendieck, \textit{ \'El\'ements de g\'eom\'etrie
    alg\'ebrique. IV.  \'Etude locale des sch\'emas et des morphismes 
de sch\'emas IV},    Inst. Hautes \'Etudes Sci. Publ. Math. No. {\bf
32} (1967). 

\item[{[HHK1]}] D. Harbater, J. Hartmann and D. Krashen,
  \textit{Applications of patching to quadratic forms and centrals
    simple algebras}, Invent. Math. {\bf 178} (2009), 231-263.

\item[{[HHK2]}] D. Harbater, J. Hartmann and D. Krashen,
  \textit{Local-global principles for torsors over arithmetic curves}, 
arxive:1108.3323v2

\item[{[HHK3]}] D. Harbater, J. Hartmann and D. Krashen,
\textit{Weierstrass preparation and algebraic invariants},
arxive:1109.6362v2

\item[{[HB]}] D.R. Heath-Brown, \textit{Zeros of systems of p-adic
    quadratic forms},   Compos. Math. {\bf 146} (2010),  271–287.
 
 \item[{[K1]}] K. Kato,  \textit{ Galois cohomology of complete
     discrete valuation fields},  Algebraic K-theory, Part II (Oberwolfach, 1980), pp. 215–238, 
Lecture Notes in Math., 967, Springer, Berlin-New York, 1982. 

 \item[{[K2]}] K. Kato, \textit{A Hasse principle for two dimensional
global fields},  J. reine Angew. Math. {\bf 366} (1986),  142-181.  
 
\item[{[KMRT]}] M.-A. Knus, A.S. Merkurjev, M. Rost and J.-P. Tignol,
  \textit{The Book of Involutions}, A.M.S, Providence RI, 1998. 

\item[{[MMW]}] P. Mammone, R. Moresi  and A. R. Wadsworth, \textit{
u-invariants of fields of characteristic 2}, Math. Z. {\bf 208} (1991), 335-347.

\item[{[M]}] A.S.Merkurjev, \textit{On the norm residue symbol of
    degree 2},  Dokl. Akad. Nauk. SSSR {\bf 261} (1981), 542-547. 
 
\item[{[O]}] M. Ojanguren, \textit{A splitting theorem for quadratic
    forms}, Comment. Math. Helvetici {\bf 57} (1982),  145-157

\item[{[L]}] T.Y.Lam, \textit{ Introduction to Quadratic Forms over
    Fields},  Grad. Stud. Math. {\bf 67}, Amer. Math. Soc., Providence, RI, 2005.

\item[{[LPS]}]  Max Lieblich, R. Parimala  and  V.  Suresh, 
  \textit{ Colliot-Th\'el\`ene's conjecture and finiteness of $u$-invariants},  preprint 2012. 

\item[{[Le]}]  D. B. Leep, \textit{The u-invariant of p-adic function
    fields},  to appear in Journal fur die reine und angewandte Mathematik.

\item[{[PS1]}] R. Parimala and V. Suresh, \textit{ Isotropy  of
    quadratic forms over function fields  in one variable over
    $p$-adic fields},   Publ. de I.H.E.S. {\bf 88} (1998), 129-150.

\item[{[PS2]}] R. Parimala and V. Suresh, \textit{The $u$-invariant of
    the function fields of $p$-adic  curves},  Annals of Mathematics
  {\bf 172} (2010),  1391-1405.

\item[{[S1]}]  D.J. Saltman, \emph{Division Algebras over $p$-adic
curves},  J. Ramanujan Math. Soc. {\bf 12} (1997), 25-47.

\item[{[S2]}] D.J. Saltman,  \emph{Cyclic Algebras over $p$-adic curves},
J. Algebra {\bf 314} (2007) 817-843.
 
\item[{[S3]}] D.J. Saltman, \emph{ Bad characteristic},
 private notes.

\item[{[Su]}]  V. Suresh,  \emph{Galois cohomology in degree 3 of
    function fields  of curves over number fields},  J. Number Theory
  {\bf 107} (2004), no. 1, 80-94.

\end{enumerate}

\end{document}